\documentclass[11pt]{amsart}

\usepackage{a4wide}
\usepackage{amssymb}
\usepackage{mathrsfs}
\usepackage{mathtools}
\usepackage{enumerate}
\usepackage{titletoc}
\usepackage[colorlinks=true, urlcolor=blue, linkcolor=blue, citecolor=green]{hyperref}
\usepackage[nameinlink]{cleveref}
\usepackage{mathscinet}
\usepackage{esint}

\theoremstyle{plain}
\newtheorem{theorem}{Theorem}[section]
\newtheorem{lemma}[theorem]{Lemma}
\newtheorem{corollary}[theorem]{Corollary}
\newtheorem{proposition}[theorem]{Proposition}
\newtheorem*{conjecture*}{Conjecture}
\newtheorem*{question*}{Question}

\theoremstyle{definition}
\newtheorem{definition}[theorem]{Definition}
\newtheorem{remark}[theorem]{Remark}

\numberwithin{equation}{section}

\makeatletter
\@namedef{subjclassname@2020}{\textup{2020} Mathematics Subject Classification}
\makeatother

\title[]{A priori estimates for solutions of degenerate fully nonlinear elliptic equations with $L^p$ data}

\author{Hongsoo Kim}
\address{Department of Mathematical Sciences, Seoul National University, Seoul 08826, Republic of Korea}
\email{rlaghdtn98@snu.ac.kr}

\author{Se-Chan Lee}
\address{School of Mathematics, Korea Institute for Advanced Study, Seoul 02455, Republic of Korea}
\email{sechan@kias.re.kr}



\thanks{Se-Chan Lee has been supported by the KIAS Individual Grant (No. MG099001) at Korea Institute for Advanced Study. H. Kim was supported by the National Research Foundation of Korea(NRF) grant funded by the Korea government [Grant No. 2022R1A2C1009312].}

\begin{document}

\begin{abstract}
We establish a priori regularity estimates for viscosity solutions of degenerate fully nonlinear elliptic equations with integrable right-hand sides. When the nonhomogeneous term belongs to $L^p$ with $p>n$, we prove optimal interior $C^{1,\alpha}$ estimates. In the critical case, we obtain a log-Lipschitz modulus of continuity under the Lorentz condition $f\in L^{n,1}$. We utilize sliding paraboloid or cusp methods to develop uniform H\"older estimates for equations that are elliptic only in suitable gradient regimes. Finally, we establish an approximation lemma for integrable right-hand sides via a corrector argument, which allows us to deduce the corresponding Schauder-type estimates.
\end{abstract}

\maketitle

%
%

\section{Introduction}
This paper is concerned with the regularity of viscosity solutions $u$ of the following degenerate fully nonlinear elliptic equation
\begin{equation*}
    |Du|^{\gamma}F(D^2u)=f \quad \text{in $B_1 \subset \mathbb{R}^n$,}
\end{equation*}
where $\gamma > 0$ and $F$ is a uniformly elliptic fully nonlinear operator. In the classical uniformly elliptic case $\gamma=0$ and $f=0$, the celebrated Krylov--Safonov theory~\cite{KS79, KS80} yields that $u \in C^{1, \alpha_0}_{\mathrm{loc}}(B_1)$ for some  $\alpha_0=\alpha_0(n, \lambda,\Lambda) \in (0,1)$ which will be fixed throughout the paper; see Theorem~\ref{thm-KS} for details.
Moreover, by the classical Evans--Krylov theory~\cite{Eva82, Kry82, Kry83}, if $F$ is convex, then $u \in C^{1,1}_{\mathrm{loc}}(B_1)$ (in fact, $C^{2,\alpha}_{\mathrm{loc}}(B_1)$); thus, we can take $\alpha_0=1$ in the convex setting.
In the degenerate case $\gamma>0$, since the coefficient $|Du|^{\gamma}$ may vanish at critical points of $u$, the equation is no longer uniformly elliptic. Nevertheless, it is by now well understood that viscosity solutions $u$ still enjoy interior $C^{1, \alpha}$ estimates in terms of the $L^{\infty}$-norm of $f$. Imbert--Silvestre~\cite{IS13} established the $C^{1, \alpha}$ regularity result for some universal exponent $\alpha \in (0,1)$ with the uniform estimate
\begin{equation*}
    \|u\|_{C^{1,\alpha}(\overline{B_{1/2}})}\leq C(\|u\|_{L^{\infty}(B_1)}+\|f\|_{L^{\infty}(B_1)}^{\frac{1}{1+\gamma}}).
\end{equation*}
Later, Ara\'ujo--Ricarte--Teixeira~\cite{ART15} captured the sharp range of gradient H\"older exponents, namely,
\begin{equation*}
    \alpha \in (0, \alpha_0) \cap \left(0, \frac{1}{1+\gamma}\right].
\end{equation*}
 We also refer to \cite{APPT22, BBLL24int, BKO25, BPRT20, DF21, FVZ21, Nas25, dSR20} for interior regularity results for equations with different degenerate structures, \cite{AS23, BBLL24, BD14, BKO252, BV21} for boundary regularity results under Dirichlet, Neumann, or oblique conditions, \cite{AJ25, LLY24, LLYZ25} for the parabolic counterparts, and \cite{ART17, HPRS24, dSV21, dSV21b} for the associated obstacle-type problems.\newline

The purpose of this paper is to develop \emph{a priori} estimates for solutions in terms of the nonhomogeneous term $f$ only through its integrability norm. Let us begin with a simple observation that suggests a natural guess for the interior regularity results in the case of $f \in L^p$ with $p \in [n, \infty)$. Consider the model equation
	\begin{equation*}
	    |Du|^{\gamma} \Delta u=f
	\end{equation*}
	and let $u(x)=|x|^{1+\alpha}$. Then, up to a multiplicative constant, we have
	\begin{equation*}
	    f(x)\simeq |x|^{\alpha(1+\gamma)-1}.
	\end{equation*}
	Thus, it is immediate to check that $f\in L^p$ if and only if
    \begin{equation*}
        \alpha>\frac{1-n/p}{1+\gamma}.
    \end{equation*}
	This computation shows that one cannot expect $C^{1, \alpha}$ regularity of solutions for an exponent larger than
	$(1-n/p)/(1+\gamma)$ under the sole assumption that $f\in L^p$. Indeed, this exponent recovers the optimal value $1/(1+\gamma)$ obtained in \cite{ART15} when $p \to \infty$, while the borderline case $p=n$ corresponds to the critical regime where one cannot expect a positive H\"older exponent for the gradient in general. We also note that a solution $u$ of $F(D^2u)=f \in L^n$ need not be Lipschitz continuous, while $u$ is $C^{\alpha}$ for any $\alpha \in (0,1)$; see \cite[Remark~5]{Tei14} for details. 
    
    On the other hand, it would also be interesting to understand the subcritical integrability range
$p<n$. Although one cannot expect any gradient H\"older regularity in this regime, it is natural to ask whether solutions still satisfy a uniform H\"older
estimate under the assumption $f \in L^p$, possibly for a suitable range of $p<n$. We leave this question for future work and refer to the related works of Escauriaza~\cite{Esc93} and Teixeira~\cite{Tei14} in the uniformly elliptic case $\gamma=0$.

Before we state our main theorems, let us specify our setting more precisely. We consider viscosity solutions $u \in C(B_1)$ of  
\begin{equation*}
    |Du|^{\gamma}F(D^2u,x)=f(x) \quad \text{in $B_1$,}
\end{equation*}
where $F(\cdot, x)$ is uniformly elliptic with $F(0, x)\equiv 0$, and $f \in L^p(B_1) \cap C(B_1)$ for some $p>n$ or $f \in L^{n, 1}(B_1) \cap C(B_1)$; see Definition~\ref{def-Lorentz} for the definition of Lorentz spaces $L^{p,q}(\Omega)$. Since we allow the operator $F$ to depend on the spatial variable $x$, we impose a small oscillation condition in the spatial variable. Following the perturbative framework of Caffarelli~\cite{Caf89}, we measure the oscillation of $F$ near $x_0 \in B_1$ by
\begin{equation}\label{eq-def-beta}
    \beta_F(x, x_0)=\beta(x, x_0) \coloneqq \sup_{M \in \mathcal{S}^n \setminus \{0\}} \frac{|F(M, x)-F(M, x_0)|}{\|M\|},
\end{equation}
where $\mathcal{S}^n$ denotes the set of symmetric $n \times n$ matrices.

Throughout the paper, we work in an \textit{a priori} setting: we assume that $f$ and $\beta_F$ are H\"older continuous for some exponent in $(0,1)$. Let us clarify the role of this auxiliary regularity assumption used in the proof of Lemma~\ref{lem-approximation}. Indeed, this assumption is used only to justify the construction of the correctors and the use of classical test functions, as in the a priori framework of \cite[Section~8.2]{CC95}. The constants in the estimates below are independent of these auxiliary H\"older norms. Hence, the final estimates depend only on the relevant $L^p$- or $L^{n,1}$-quantities appearing in the hypotheses.\newline

Our first main theorem gives the optimal interior $C^{1,\alpha}$ estimate in the supercritical regime $p>n$. The admissible range of $\alpha$ is consistent with the heuristic obstruction suggested by the model example discussed above.

\begin{theorem}[$C^{1, \alpha}$ regularity]\label{thm-main-p>n}
    Suppose that $f \in L^p(B_1)\cap C(B_1)$ with $p>n$. Let $u \in C(B_1)$ be a viscosity solution of 
    \begin{equation*}
        |Du|^\gamma F(D^2u,x) =f \quad \text{ in }B_1.
    \end{equation*}
    For 
    \begin{equation*}
        \alpha \in(0,\alpha_0)\cap \left(0,\frac{1-n/p}{1+\gamma} \right],
    \end{equation*}
    there exists a constant $\eta>0$ depending only on $n$, $p$, $\lambda$, $\Lambda$, $\gamma$ and $\alpha$ such that if 
    \begin{equation*}
         \sup_{y \in B_{1/2}} \left(\fint_{B_r(y)} \beta^p(x, y) \, dx \right)^{1/p}  \leq \eta \quad \text{ for any $r \in (0, 1/2)$,}
    \end{equation*}
     then $u \in C_{\mathrm{loc}}^{1, \alpha}(B_1)$ with the uniform estimate
    \begin{equation*}
        \|u\|_{C^{1,\alpha}(B_{1/2})} \leq C(\|u\|_{L^\infty(B_1)}+\|f\|_{L^p(B_1)}^{\frac{1}{1+\gamma}}),
    \end{equation*}
    where $C>0$ is a constant depending only on $n$, $p$, $\lambda$, $\Lambda$, $\gamma$ and $\alpha$.
\end{theorem}

We next move our attention to the critical case $p=n$, where we cannot expect $C^{1, \alpha}$ regularity of viscosity solutions due to the earlier observation. Instead, we describe a borderline modulus of continuity with the aid of the Lorentz space $L^{n, 1}$, which is slightly stronger than $L^n$ but still contains $L^p$ for every $p>n$ on bounded domains.

\begin{theorem}[Log-Lipschitz regularity]\label{thm-main-p=n}
Suppose that $f \in L^{n, 1}(B_1) \cap C(B_1)$. Let $u \in C(B_1)$ be a viscosity solution of 
    \begin{equation*}
        |Du|^\gamma F(D^2u,x) =f \quad \text{ in }B_1.
    \end{equation*}
    Then there exists $\eta>0$ depending only on $n$, $\lambda$, $\Lambda$ and $\gamma$ such that if
    \begin{equation*}
         \sup_{y \in B_{1/2}} \frac{1}{r}\|\beta(\cdot, y)\|_{L^{n,1}(B_{r}(y))}  \leq \eta \quad \text{ for any $r \in (0, 1/2)$,}
    \end{equation*}
    then $u$ is locally log-Lipschitz, i.e., for any $x,y \in B_{1/2}$,
    \begin{equation*}
        |u(x)-u(y)| \leq C(\|u\|_{L^\infty(B_1)}+\|f\|_{L^{n,1}(B_1)}^{\frac{1}{1+\gamma}})\cdot \omega(|x-y|),
    \end{equation*}
    where $\omega(t) = -t\log t$ and $C>0$ is a constant depending only on $n$, $\lambda$, $\Lambda$ and $\gamma$.
\end{theorem}

Let us briefly summarize the strategy of the proof. A key step is to obtain a uniform H\"older estimate for a class of equations (Theorem~\ref{holder}) that arise in the iteration argument. To be precise, after subtracting an affine function from $u$ and scaling, one is led to investigate equations of the form 
\begin{equation*}
    |Du-q|^{\gamma}F(D^2u, x)=f(x),
\end{equation*}
or, in view of Pucci extremal operators (see Definition~\ref{def-pucci}),
\begin{equation*}
     \left\{
             \begin{aligned}
                 \mathcal{M}^-(D^2u) &\leq |f|\\
                  \mathcal{M}^+(D^2u) &\geq -|f| 
             \end{aligned}
             \right.
             \quad \text{ in } \{|Du-q|>1\} \cap B_1.
\end{equation*}
Then Proposition~\ref{twoprop} shows that $u$ satisfies one of two alternatives, corresponding to the regimes of ``small gradients'' ($|Du|<A_0$) and ``large gradients'' ($|Du|>A_0+2$). Thus, it suffices to prove a uniform $L^{\varepsilon}$ estimate, which is independent of $q \in \mathbb{R}^n$, by combining the measure-type estimate and the doubling property in each regime. At this stage, we use the sliding paraboloid or cusp method to exploit the information on the size of gradients. Such sliding methods can be found in Savin~\cite{Sav07}, Mooney~\cite{Moo15, Moo19}, Imbert--Silvestre~\cite{IS16} and Byun--Kim--Oh~\cite{BKO252}. We also note that all results in Section~\ref{sec3} remain valid under the weaker assumption $f \in L^n$.

For $f\in L^{\infty}$, Imbert--Silvestre~\cite{IS13} obtained a uniform modulus of continuity result independent of $q \in \mathbb{R}^n$, by proving Lipschitz estimates (for large $|q|$'s) via the Ishii--Lions method. Nevertheless, the pointwise nature of Ishii--Lions method prevents its direct application of to merely integrable right-hand sides. Therefore, we replace this approach with a measure-theoretic one based on sliding paraboloids and cusps, which is robust under the assumption $f\in L^n$.

Once we develop the compactness result, we show the improvement of flatness or approximation lemma (Lemma~\ref{lem-approximation}) to guarantee that the solution $u$ is close to a solution of a frozen, homogeneous uniformly elliptic equation. It is noteworthy that, at this step, another difficulty arises from the lack of pointwise control on $f$. To be precise, when we try to verify the stability issue in the proof of approximation lemma, it is necessary to construct appropriate corrector functions $\psi$ that converge to $0$ in the Lipschitz sense. In order to ensure the convergence, we apply the borderline potential estimate for the Lorentz space $L^{n,1}$ developed by Daskalopoulos--Kuusi--Mingione~\cite{DKM14}, and so the current proof for the approximation lemma seems not valid under a weaker condition $f \in L^n$. Then the remaining iteration arguments are rather standard; the approximation lemma together with the smallness assumption on $\beta$ leads to the desired $C^{1, \alpha}$ estimate and log-Lipschitz estimate depending on $f$.\newline

The paper is organized as follows. In Section~\ref{sec-preliminaries}, we collect the basic definitions and preliminary results concerning viscosity solutions, uniformly elliptic operators, Pucci extremal operators and Lorentz spaces. In Section~\ref{sec3}, we establish uniform H\"older estimates for elliptic equations that hold only where $|Du-q|>1$. Section~\ref{sec-improvement} is devoted to the improvement of flatness via a compactness argument and perturbation from frozen homogeneous equations. In Section~\ref{sec4}, we prove the $C^{1,\alpha}$ regularity result in the case $p>n$. Finally, in Section~\ref{sec5}, we deal with the critical case by proving the log-Lipschitz estimate under the Lorentz assumption $f \in L^{n, 1}$.

\section{Preliminaries}\label{sec-preliminaries}
Let $G: \mathcal{S}^n \times \mathbb{R}^n \times \mathbb{R} \times \mathbb{R}^n \to \mathbb{R}$ be a continuous map. We are concerned with viscosity solutions $u$ of uniformly elliptic fully nonlinear equations
\begin{equation}\label{eq-general}
    G(D^2u, Du, u, x)=0,
\end{equation}
and, in particular, we may set
\begin{equation*}
    G(D^2u, Du, u, x) = |Du-q|^{\gamma}F(D^2u, x)-f(x).
\end{equation*}

\begin{definition}[Viscosity solutions]
	A continuous function $u$ is called a \textit{viscosity supersolution} of \eqref{eq-general} if for all $x_0\in \Omega$ and $\varphi\in C^{2}(\Omega)$ such that $u-\varphi$ has a local minimum at $x_0$, then 
	\begin{equation*}
		G(D^{2}\varphi(x_0), D\varphi(x_0), u(x_0), x_0)\leq 0.
	\end{equation*}
	A continuous function $u$ is called a \textit{viscosity subsolution} of \eqref{eq-general} if for all $x_0\in \Omega$ and $\varphi\in C^{2}(\Omega)$ such that $u-\varphi$ has a local maximum at $x_0$, there holds 
	\begin{equation*}
		G(D^{2}\varphi(x_0), D\varphi(x_0), u(x_0), x_0)\geq 0.
	\end{equation*}
	We say that $u\in C(\Omega)$ is a \textit{viscosity solution} of \eqref{eq-general} if $u$ is a viscosity supersolution and a subsolution simultaneously.
\end{definition}

\begin{definition}[Uniformly elliptic operator]
    We say that $F$ is \emph{uniformly elliptic} if there exist two positive constants $\lambda \leq \Lambda$ such that for any $M \in \mathcal{S}^n$ and $x \in \Omega$,
    \begin{equation*}
        \lambda \mathrm{tr}\,N \leq F(M+N, x)-F(M, x) \leq \Lambda \mathrm{tr}\,N  \quad \text{for any $N \geq 0$}.
    \end{equation*}
    Here, $N \geq 0$ means that $N$ is a nonnegative definite symmetric matrix.
\end{definition}

\begin{definition}[Pucci's extremal operators]\label{def-pucci}
    For given ellipticity constants $0<\lambda \leq \Lambda$, we define \emph{Pucci's extremal operators} by
    \begin{equation*}
        \mathcal{M}^+(M)=\mathcal{M}_{\lambda, \Lambda}^+(M)=\Lambda\sum_{e_i \geq 0}e_i+\lambda \sum_{e_i<0}e_i
    \end{equation*}
    and
     \begin{equation*}
        \mathcal{M}^-(M)=\mathcal{M}_{\lambda, \Lambda}^-(M)=\lambda\sum_{e_i \geq 0}e_i+\Lambda \sum_{e_i<0}e_i,
    \end{equation*}
    where $e_i=e_i(M)$ are the eigenvalues of $M \in \mathcal{S}^n$. 
\end{definition}

Let us recall the interior $C^{1, \alpha}$ regularity result for uniformly elliptic fully nonlinear equations, which follows from the Krylov--Safonov theory~\cite{KS79, KS80}.
\begin{theorem}[{\cite[Corollary~5.7]{CC95}}]\label{thm-KS}
    Let $F:\mathcal{S}^n \to \mathbb{R}$ be a $(\lambda, \Lambda)$-uniformly elliptic operator. If $u \in C(B_1)$ is a viscosity solution of 
    \begin{equation*}
        F(D^2u)=0 \quad \text{in $B_1$},
    \end{equation*}
    then there exists a constant $\alpha_0 \in (0,1)$, depending only on $n$, $\lambda$ and $\Lambda$, such that $u \in C^{1, \alpha_0}(\overline{B_{1/2}})$ and
    \begin{equation*}
        \|u\|_{C^{1, \alpha_0}(\overline{B_{1/2}})} \leq C\|u\|_{L^{\infty}(B_1)},
    \end{equation*}
    where $C>0$ is a constant depending only on $n$, $\lambda$ and $\Lambda$.
\end{theorem}

\begin{remark}
    If we additionally assume that $F$ is convex (or concave), then the Evans--Krylov theory (see \cite[Section~6]{CC95} for instance) says that a viscosity solution $u \in C(B_1)$ of $F(D^2u)=0$ in $B_1$ belongs to  $C^{1,1}_{\mathrm{loc}}(B_1)$ (in fact, $ C_{\mathrm{loc}}^{2, \tilde \alpha_0}(B_1)$ for some $\tilde \alpha_0 \in (0,1)$), so we can take $\alpha_0=1$ when $F$ is convex.
\end{remark}

Let us end this section with the definition of Lorentz spaces, which is closely related to several borderline regularity results.
\begin{definition}[Lorentz space]\label{def-Lorentz}
Let $1\leq p<\infty$, $0<q<\infty$, and let $\Omega\subset\mathbb{R}^n$ be an open set. For a measurable function $g:\Omega\to\mathbb{R}$, we say that $g\in L^{p,q}(\Omega)$ if
\begin{equation*}
    \|g\|_{L^{p,q}(\Omega)}
    \coloneqq \left(\int_0^\infty\left(t^p\left|\{x\in\Omega: |g(x)|>t\}\right|\right)^{q/p}\frac{dt}{t}\right)^{1/q}.
\end{equation*}
\end{definition}

We recall that
\begin{equation*}
    L^{p,p}(\Omega)=L^p(\Omega)
\end{equation*}
for every $1\leq p<\infty$, with equivalent norms. Moreover, if $|\Omega|<\infty$, then
\begin{equation*}
    L^p(\Omega)\subset L^{n,1}(\Omega)\subset L^n(\Omega) \quad \text{and} \quad \|f\|_{L^n(\Omega)} \lesssim \|f\|_{L^{n, 1}(\Omega)} \lesssim\|f\|_{L^p(\Omega)}
\end{equation*}
for every $p>n$. 

\section{Uniform H{\"o}lder estimates of elliptic equations that hold only where \texorpdfstring{$|Du-q|>1$}{|Du-q|>1}} \label{sec3}
In this section, we seek uniform H{\"o}lder estimates for viscosity solutions of 
\begin{equation*}
    |Du-q|^\gamma F(D^2u, x) =f \quad \text{in $B_1$},
\end{equation*}
where $q \in \mathbb{R}^n$ and $f\in L^n(B_1) \cap C(B_1)$. Our goal is to derive the uniform estimate, which is an $L^n$ analogue of \cite[Lemma~3]{IS13}, by employing the fact that this equation is uniformly elliptic when $|Du-q|>1$. 

\begin{theorem}[Modulus of continuity independent of $q$]\label{lem-equicontinuity}
    Let $u \in C(B_1)$ be a viscosity solution of
    \begin{equation*}
        |Du-q|^{\gamma}F(D^2u, x)=f \quad \text{in $B_1$},
    \end{equation*}
    where $q \in \mathbb{R}^n$ and $f \in L^n(B_1) \cap C(B_1)$.

    If $\|u\|_{L^{\infty}(B_1)} \leq 1$ and $\|f\|_{L^n(B_1)} \leq 1$, then $u$ satisfies a uniform $C^{\beta}$ estimate for some universal exponent $\beta \in (0,1)$, i.e.,
    \begin{equation*}
        [u]_{C^{\beta}(\overline{B_{1/2}})} \leq C,
    \end{equation*}
    where $C$ is independent of $q \in \mathbb{R}^n$.
\end{theorem}

    Here we observe that if $u \in C(B_1)$ satisfies 
    \begin{equation*}
        |Du-q|^{\gamma}F(D^2u, x)=f \quad \text{in $B_1$},
    \end{equation*}
    then $u$ satisfies
     \begin{equation*}
     \left\{
             \begin{aligned}
                 \mathcal{M}^-(D^2u) &\leq |f|\\
                  \mathcal{M}^+(D^2u)&\geq -|f| 
             \end{aligned}
             \right.
             \quad \text{ in } \{|Du-q|>1\} \cap B_1 
    \end{equation*}
    in the viscosity sense. Thus, we are going to prove the uniform estimates for a larger class of solutions, and Theorem~\ref{lem-equicontinuity} is a direct corollary of Theorem~\ref{holder} below.

\begin{theorem} \label{holder}
    Let $u \in C(B_1)$ satisfy the following inequalities:
    \begin{equation*}
              \mathcal{M}^-(D^2u)-\Lambda|Du| \leq f \leq \mathcal{M}^+(D^2u)+\Lambda|Du| \quad \text{ in } \{|Du-q|>1\} \cap B_1 
    \end{equation*}
    for some $q \in \mathbb{R}^n$ and $f \in L^n(B_1) \cap C(B_1)$.
    Then $u\in C^{\beta}_{\mathrm{loc}}(B_{1})$ for some $\beta>0$, depending only on $n$, $\lambda$ and $\Lambda$, with the uniform estimate 
    \begin{equation*}
    \|u\|_{C^{\beta}(B_{1/2})} \leq C(n,\lambda,\Lambda,\|u\|_{L^\infty(B_1)}, \|f\|_{L^n(B_1)}),
    \end{equation*}
    which is independent of $q \in \mathbb{R}^n$.
\end{theorem}

The proof follows an argument similar to that in \cite{BKO252}, adapted here for the case of $L^n$ data. We begin by observing that the scaled function of the solution $u$ satisfies one of two alternatives (see \cite[Proposition~3.1]{BKO252} for details).

\begin{proposition} \label{twoprop}
    Let  $u\in C(B_1)$ satisfy
    \begin{equation*}
        \mathcal{M}^-(D^2u)-\Lambda|Du| \leq f \ \text{ in } \{|Du-q|> 1\} \cap B_1,
    \end{equation*}
    for some $q\in \mathbb{R}^n$.
    Then for any $A_0 >0$, the scaled function $v(x) = u(rx)/M$ with $M\geq 1$ and $r \in (0, 1]$ satisfies either
    \begin{equation*}
        \mathcal{M}^-(D^2v)-\Lambda|Dv| \leq \tilde{f} \ \text{ in } \{|Dv| \leq A_0\} \cap B_{1/r}
    \end{equation*}
    or
    \begin{equation*}
        \mathcal{M}^-(D^2v)-\Lambda|Dv| \leq \tilde{f}  \ \text{ in } \{|Dv| \geq A_0+2\} \cap B_{1/r},
    \end{equation*}
    where $\tilde{f}(x)=r^2f(rx)/M$. Moreover, $\|\tilde{f}\|_{L^n(B_{1/r})} \leq \|f\|_{L^n(B_1)}$.
\end{proposition}

In view of Proposition~\ref{twoprop}, we investigate the regularity in each of the two regimes: the small-gradient regime and the large-gradient regime.

Let us first prove a measure estimate for the equation, which is uniformly elliptic in the regime of small gradients, by applying the sliding paraboloid method.

\begin{lemma}\label{A00}
    There exist small $\delta_0>0$, $\varepsilon_0>0$ and large $K>1$, $A_0>1$ such that if $u \in C(B_1)$ satisfies
    \begin{equation*}
    \left\{
    \begin{aligned}
        u &\geq 0 &&\text{in } B_1,\\
        \mathcal{M}^-(D^2u)-\Lambda|Du| &\leq f &&\text{in } \{|Du| \leq A_0\}\cap B_1, \\
        \|f\|_{L^n(B_1)} &\leq \delta_0,\\
        |\{u>K\} \cap B_1| &\geq (1-\varepsilon_0)|B_1|,
    \end{aligned}
    \right.
    \end{equation*}
    then $u>1$ in $B_{1/4}$.
\end{lemma}

\begin{proof}
By a standard approximation argument using inf-convolutions (see \cite[Section~3]{IS16} for instance), we may assume that $u\in C^2(B_1)\cap C(\overline{B_1})$.
Suppose, toward a contradiction, that the conditions of the Lemma hold but $u(x_0) \leq 1$ for some $x_0 \in B_{1/4}$.

Consider the vertex set $V = B_{1/4}$ and we slide a paraboloid $\phi(z-y) = -10|z-y|^2$ with vertex $y \in V$ from below, until it touches the graph of $u$ for the first time at a point $x \in \overline{B_1}$.
Let $T \subset \overline{B_1}$ be the set of these contact points.
For $y \in V$, we observe that
\begin{equation*}
    u(x_0) + 10|x_0-y|^2 < 1+ 10 \left(\frac{1}{2}\right)^2=\frac{7}{2}, \quad \text{since $x_0 \in B_{1/4}$}.
\end{equation*}
If $x \in \partial B_1$, then since $u \geq 0$ and $y \in B_{1/4}$, we find that
\begin{equation*}
    u(x) + 10|x-y|^2 > 10 \left(1-\frac{1}{4}\right)^2>\frac{7}{2}.
\end{equation*}
Therefore, we conclude that the minimum of $u-\phi(\cdot-y)$ cannot be attained on the boundary $\partial B_1$ and $T \subset \{ u<K\} \cap B_1$ for $K=7/2$.
Note that for any $x \in T$, we have
\begin{align}\label{Ducond1}
    Du(x) &= D\phi(x-y) = 20(y-x),\\ \label{D2ucond1}
    D^2u(x) &\geq D^2\phi(x-y) = -20I.
\end{align}
Choosing $A_0=40 \geq \sup_{ B_2}|D\phi|$, we ensure $T \subset \{|Du| \leq A_0\}$, and so the inequality 
\begin{equation*}
    \mathcal{M}^-(D^2u)-\Lambda|Du| \leq f
\end{equation*}
is satisfied in $T$. By differentiating \eqref{Ducond1} with respect to $x$, we obtain
\begin{equation*}
    D_xy = I + \frac{1}{20}D^2u(x) \geq 0,
\end{equation*}
where we used \eqref{D2ucond1}. Using the properties of $\mathcal{M}^-$, we have for $x \in T$,
\begin{equation*}
    \begin{aligned}
        \lambda \mathrm{tr}\,D_xy &= \mathcal{M}^-(D_xy) 
   \leq  \mathcal{M}^+(I) + \frac{1}{20}\mathcal{M}^-(D^2u) \\
   &\leq n\Lambda + \frac{1}{20}(f(x) + \Lambda|Du|) \\
   &\leq C(1+|f(x)|)
    \end{aligned}
\end{equation*}
for some $C=C(n, \Lambda)>0$. Moreover, by the inequality $\det M \leq \left( \frac{\mathrm{tr}\, M}{n}\right)^n$ for any nonnegative matrix $M \geq0$, we obtain
\begin{equation*}
	\det D_xy\leq \left(\frac{\mathrm{tr}\,{D_xy} }{n}\right)^n \leq C(1 +|f(x)|)^n.
\end{equation*}
Applying the area formula and recalling that $|T| \leq |\{ u<K\} \cap B_1| \leq \varepsilon_0|B_1| $,
\begin{equation*}
    |B_{1/4}| = |V| =\int_{T}|\det D_xy|\,dx \leq C(|T|+\|f\|_{L^n(T)}^n) \leq C(\varepsilon_0|B_1| +\delta_0^n),
\end{equation*}
which leads to a contradiction if we choose $\delta_0$ and $\varepsilon_0$ small enough.
\end{proof}

We next prove the doubling property for lower bounds of supersolutions by constructing an appropriate barrier function.
\begin{lemma}\label{A01}
    There exist small $\delta_0>0$ and large $K>1$, $A_0>1$ such that if  $u \in C(B_3)$ satisfies
    \begin{equation*}
    \left\{
    \begin{aligned}
        u &\geq 0 &&\text{in } B_3,\\
        \mathcal{M}^-(D^2u)-\Lambda|Du| &\leq f &&\text{in } \{|Du| \leq A_0\}\cap B_3, \\
        \|f\|_{L^n(B_3)} &\leq \delta_0, \\
        u&>K &&\text{in } B_{1/4},
    \end{aligned}
    \right.
    \end{equation*}
    then $u>1$ in $B_1$.
\end{lemma} 

\begin{proof}
By standard approximation again, we may assume that $u\in C^2(B_3)\cap C(\overline{B_3})$. We use the contradiction by assuming that $u(x_0) \leq 1$ for some $x_0 \in B_{1}$.
We define the following barrier function
\begin{equation*}
    \Phi(z) =  \frac{M}{p}(|z|^{-p}-2^{-p}) \quad \text { in } \mathbb{R}^n\setminus B_{1/8},
\end{equation*}
for some $p>1$ and $M>1$ to be determined later, and extend $\Phi (>0)$ smoothly inside $B_{1/8}$. For $z \in B_4\setminus B_{1/8}$, a direct calculation shows
\begin{equation} 
\begin{aligned}\label{barrierPDE}
    \mathcal{M}^-(D^2\Phi) -\Lambda|D\Phi| &= M |z|^{-p-2}(\lambda(p+1)- \Lambda(n-1) - \Lambda|z|) \\
    &\geq 1, 
\end{aligned}    
\end{equation}
provided that $p=p(n,\lambda,\Lambda)>1$ and $M=M(n,\lambda,\Lambda)>1$ are sufficiently large.

We choose $M$ further larger to ensure that $\Phi >2$ in $B_{3/2}$.
Consider the vertex set $V = B_{1/8}$ and we slide $\Phi(\cdot-y)$ with vertex $y \in V$ from below, until it touches the graph of $u$ for the first time at a point $x \in \overline{B_3}$. Let $T$ denote the set of corresponding contact points.

For $y \in V$, we observe that
\begin{equation*}
    u(x_0)-\Phi(x_0-y) < 1-2=-1, \quad \text{since $x_0 \in B_1$}.
\end{equation*}
If $x\in \partial B_3$, then since $u \geq 0$  and $\partial B_{3} \subset \{z:\Phi(z-y) \leq 0\}$, we find that
\begin{equation*}
    u(x) - \Phi(x-y) \geq 0.
\end{equation*}
Therefore, we conclude that the minimum of $u-\Phi(\cdot-y)$ cannot be attained on the boundary $\partial B_3$ and $T \subset \{ u<K\} \cap B_3$ for $K=\|\Phi\|_{L^\infty(\mathbb{R}^n)}$.
Recalling that $u>K$ in $B_{1/4}$, we have $T \subset \{ u<K\} \cap (B_3 \setminus B_{1/4})$.
Moreover, we find $1/8<|x-y| < 4$ for $x \in T$ and $y \in V$.
Note that for any $x \in T$, we get
\begin{align} \label{Ducond2}
    Du(x) &= D\Phi(x-y),\\ \label{D2ucond2}
    D^2u(x) &\geq D^2\Phi(x-y).
\end{align}
Choosing $A_0=\sup_{B_4\setminus B_{1/8}}|D\Phi|$, we obtain $T \subset \{|Du| \leq A_0\}$, and so the inequality
\begin{equation*}
    \mathcal{M}^-(D^2u)-\Lambda|Du| \leq f
\end{equation*}
is satisfied in $T$. Therefore, for $x \in T$, we have
\begin{equation} \label{uequ}
    1 \leq \mathcal{M}^-(D^2\Phi) -\Lambda|D\Phi| \leq \mathcal{M}^-(D^2u) -\Lambda|Du| \leq f(x).
\end{equation}
By combining \eqref{D2ucond2} and \eqref{uequ}, we find
\begin{equation*}
    |D^2u(x)| \leq C(|D^2\Phi(x-y)| + |Du(x)| +f(x)) \leq Cf(x),
\end{equation*}
where we used $|D^2\Phi(x-y)| \leq C=C(n, \lambda, \Lambda)$.
By differentiating \eqref{Ducond2} with respect to $x$ together with the fact that $D^2\Phi(\cdot)$ is (uniformly) invertible in $B_4 \setminus B_{1/8}$, we get
\begin{equation*}
	D_xy= I- (D^2 \Phi)^{-1}(x-y)  \cdot D^2u(x),
\end{equation*}
which implies that
\begin{equation*}
   |\det D_xy| \leq C|f(x)|^n.
\end{equation*}
Applying the area formula, we arrive at
\begin{equation*}
    |B_{1/8}| = |V| =\int_{T}|\det D_xy|\,dx \leq C\|f\|_{L^n(T)}^n \leq C\delta_0^n.
\end{equation*}
Choosing $\delta_0$ sufficiently small yields a contradiction.
\end{proof}

Combining Lemmas~\ref{A00} and \ref{A01}, we obtain the following corollary.
\begin{corollary} \label{A0}
    There exist small $\delta_0>0$, $\varepsilon_0>0$ and large $K>1$, $A_0>1$ such that if  $u \in C(B_3)$ satisfies
    \begin{equation*}
    \left\{
    \begin{aligned}
        u &\geq 0 &&\text{in } B_3,\\
        \mathcal{M}^-(D^2u)-\Lambda|Du| &\leq f &&\text{in } \{|Du| \leq A_0\}\cap B_3, \\
        \|f\|_{L^n(B_3)} &\leq \delta_0, \\
        |\{u>K\} \cap B_1| &\geq (1-\varepsilon_0)|B_1|,
    \end{aligned}
    \right.
    \end{equation*}
    then $u>1$ in $B_1$.
\end{corollary}

\begin{proof}
    Let $(\delta_1,K_1,A_1)$ and $(\delta_2,K_2,A_2)$  be the constants from Lemmas \ref{A00} and \ref{A01}, respectively.
    We set $K=K_1 K_2$, $A_0 = \max \{A_1K_2,A_2\}$ and $\delta_0=\min\{\delta_1, \delta_2\}$.
    Then the scaled function $v(x)=u(x)/K_2$ satisfies the assumptions of Lemma~\ref{A00}, implying $v>1$ in $B_{1/4}$.
    Thus, $u$ satisfies the assumptions of Lemma~\ref{A01}, and we conclude that $u>1$ in $B_1$.
\end{proof}

We next consider the equation which is uniformly elliptic when gradient is large and prove the interior measure estimate lemma similar to Lemma~\ref{A00} but using a cusp instead of a paraboloid, as in \cite{IS16}.

\begin{lemma} \label{A10}
    For any $A_1>1$, there exist small $\delta_0>0$, $\varepsilon_0>0$ and large $K>1$ such that if $u \in C(B_1)$ satisfies
    \begin{equation*}
    \left\{
    \begin{aligned}
        u &\geq 0 &&\text{in } B_1,\\
        \mathcal{M}^-(D^2u)-\Lambda|Du| &\leq f &&\text{in } \{|Du| \geq A_1\}\cap B_1, \\
        \|f\|_{L^n(B_1)} &\leq \delta_0,\\
        |\{u>K\} \cap B_1| &\geq (1-\varepsilon_0)|B_1|,            
    \end{aligned}
    \right.
    \end{equation*}
    then $u>1$ in $B_{1/4}$.
\end{lemma}

\begin{proof}
Without loss of generality, we may assume that $u \in C^2(B_1) \cap C(\overline{B_1})$ via inf-convolution. Suppose, toward a contradiction, that $u(x_0) \leq 1$ for some $x_0 \in B_{1/4}$. 

Let $V = \{u>K\} \cap B_{1/4}$. 
For a constant $C>1$ to be determined later, we define the cusp $\phi(z) = -C|z|^{1/2}$.
For each $y \in V$, we slide the cusp $\phi(\cdot-y)$ from below, until it touches the graph of $u$ for the first time at $x \in \overline{B_1}$.
Let $T$ be the set of these contact points.

By choosing $C>1$ large enough, we ensure that $T \subset \{u <K\} \cap  B_1$ for some $K=K(C)>1$ by using a similar argument in Lemma~\ref{A00}. 
We choose $C=C(n,A_1)>1$ further larger so that $A_1 \leq \inf_{B_2} |D\phi|$.
Then we have $ T \subset \{|Du| \geq A_1\}$, and the inequality 
\begin{equation*}
    \mathcal{M}^-(D^2u)-\Lambda|Du| \leq f
\end{equation*}
holds in $T$. Note that for $x \in T$ and $y \in V$, we have $x\neq y$ since $T\cap V =\varnothing$.
Moreover, we have
\begin{align} \label{Ducond3}
    Du(x) &= D\phi(x-y) =\frac{C}{2}|y-x|^{-1/2}\frac{y-x}{|y-x|},\\ \label{D2ucond3}
    D^2u(x) &\geq D^2\phi(x-y)=\frac{C}{4}|y-x|^{-3/2}\left(3\frac{y-x}{|y-x|} \otimes \frac{y-x}{|y-x|}-2I\right).
\end{align}
Therefore, using \eqref{D2ucond3} and $\mathcal{M}^-(D^2u)-\Lambda|Du| \leq f$, we obtain
\begin{equation*}
    |D^2u(x)| \leq C(|f(x)|+|D^2\phi(x-y)| + |D\phi(x-y)|).
\end{equation*}
By differentiating \eqref{Ducond3} with respect to $x$, we get
\begin{equation*}
    D_xy = I -(D^2\phi)^{-1}(x-y) \cdot D^2u(x),
\end{equation*}
which implies
\begin{equation*}
    |D_xy| \leq C\left(1+\frac{|f(x)|+|D^2\phi(x-y)| + |D\phi(x-y)|}{|D^2\phi(x-y)|}\right) \leq C(1+|f(x)|).
\end{equation*}
Applying the area formula, we obtain
\begin{equation*}
    |B_{1/4}|-\varepsilon_0|B_1| \leq  |V| =\int_{T}|\det D_xy| \, dx \leq C(|T|+\|f\|_{L^n(T)}^n) \leq C(\varepsilon_0|B_1| +\delta_0^n),
\end{equation*}
which provides a contradiction for sufficiently small $\varepsilon_0$ and $\delta_0$.
\end{proof}

We again use the barrier function in Lemma~\ref{A01} to prove the doubling property.
\begin{lemma} \label{A11}
    For any $A_1>1$, there exist small $\delta_0>0$ and large $K>1$ such that if $u \in C(B_3)$  satisfies
    \begin{equation*}
    \left\{
    \begin{aligned}
        u &\geq 0 &&\text{in } B_3,\\
        \mathcal{M}^-(D^2u)-\Lambda|Du| &\leq f &&\text{in } \{|Du| \geq A_1\}\cap B_3, \\
        \|f\|_{L^n(B_3)} &\leq \delta_0, \\
        u&>K &&\text{in } B_{1/4},
    \end{aligned}
    \right.
    \end{equation*}
    then $u>1$ in $B_1$.
\end{lemma}

\begin{proof}
The proof follows the same argument as Lemma~\ref{A01}, utilizing the barrier function:
\begin{equation*}
    \Phi(z) =  \frac{M}{p}(|z|^{-p}-2^{-p}) \quad \text { in } \mathbb{R}^n\setminus B_{1/8}.
\end{equation*}
Let  $V = B_{1/8}$ be the vertex set and $T$ be the set of corresponding contact points obtained by sliding $\Phi(z-y)$ with $y \in V$.
We choose $p=p(n, \lambda, \Lambda)>1$ and $M=M(n,\lambda,\Lambda, A_1)>1$ sufficiently large so that \eqref{barrierPDE} holds, $T \subset \{ u<K\} \cap (B_3 \setminus B_{1/4})$ with $K=\|\Phi\|_{L^\infty(\mathbb{R}^n)}$, and
\begin{equation*}
    A_1 \leq  \inf_{B_4\setminus B_{1/8}}|D\Phi|,
\end{equation*}
so that the inequality holds in $T$. The remainder of the proof proceeds exactly as in Lemma~\ref{A01}.
\end{proof}

Combining Lemmas~\ref{A10} and \ref{A11}, we also obtain the following corollary.

\begin{corollary} \label{A1}
    For any $A_1>1$, there exist small $\delta_0>0$, $\varepsilon_0>0$ and large $K>1$, such that if  $u \in C(B_3)$ satisfies
    \begin{equation*}
    \left\{
    \begin{aligned}
        u &\geq 0 &&\text{in } B_3,\\
        \mathcal{M}^-(D^2u)-\Lambda|Du| &\leq f &&\text{in } \{|Du| \geq A_1\}\cap B_3, \\
        \|f\|_{L^n(B_3)} &\leq \delta_0, \\
        |\{u>K\} \cap B_1| &\geq (1-\varepsilon_0)|B_1|,
    \end{aligned}
    \right.
    \end{equation*}
    then $u>1$ in $B_1$.
\end{corollary}

\begin{proof}
    Let $(\delta_1,K_1)$, $(\delta_2,K_2)$ be the constants from Lemmas \ref{A10} and \ref{A11} respectively.
    We choose $\delta_0 = \min\{\delta_1,\delta_2\}$ and $K=K_1 K_2$, then the corollary follows from  the same argument above.
\end{proof}

We summarize the results obtained so far by combining the preceding lemmas above.
\begin{corollary} \label{meas}
    There exist $\delta_0>0$, $\varepsilon_0>0$, $K>1$ and $A_0>1$ such that if $u \in C(B_3)$ satisfies either
    \begin{equation*}
    \left\{
    \begin{aligned}
            u &\geq 0 &&\text{in } B_3,\\
            \mathcal{M}^-(D^2u)-\Lambda|Du| &\leq f &&\text{in } \{|Du| \leq A_0\}\cap B_3,\\
            \|f\|_{L^n(B_3)} &\leq \delta_0, \\
            |\{u>K\} \cap B_1| &\geq (1-\varepsilon_0)|B_1|,
    \end{aligned}
    \right.
    \end{equation*}
    or
    \begin{equation*}
    \left\{
        \begin{aligned}
            u &\geq 0 &&\text{in } B_3,\\
            \mathcal{M}^-(D^2u)-\Lambda|Du| &\leq f &&\text{in } \{|Du| \geq A_0+2\}\cap B_3,\\
            \|f\|_{L^n(B_3)} &\leq \delta_0, \\
            |\{u>K\} \cap B_1| &\geq (1-\varepsilon_0)|B_1|,
        \end{aligned}
        \right.
    \end{equation*}   
    we have $u>1$ in $B_1$.
\end{corollary}

\begin{proof}
    We first apply Corollary~\ref{A0} and let $(\delta_1, \varepsilon_1, K_1, A_0)$ be the constants from Corollary~\ref{A0}. Then we apply Corollary~\ref{A1} for $A_1=A_0+2$ and let $(\delta_2, \varepsilon_2, K_2)$ be the constants from Corollary~\ref{A1}.Then the desired conclusion holds with the choices $\delta_0=\min\{\delta_1, \delta_2\}$, $\varepsilon_0=\min\{\varepsilon_1, \varepsilon_2\}$ and $K=\max\{K_1, K_2\}$.
\end{proof}

We are now ready to prove $L^\varepsilon$ estimate by utilizing the growing ink-spots lemma.
\begin{theorem}
    There exist $\delta_0>0$, $\varepsilon>0$ and $C>1$ such that for any $q \in \mathbb{R}^n$, if $u \in C(B_3)$ satisfies
    \begin{equation*}
    \left\{
        \begin{aligned}
            u &\geq 0 &&\text{in } B_3,\\
            \mathcal{M}^-(D^2u)-\Lambda|Du| &\leq f &&\text{in } \{|Du-q| > 1\}\cap B_3,\\
            \|f\|_{L^n(B_3)} &\leq \delta_0, \\
            \inf_{B_1} u &\leq 1,
        \end{aligned}
        \right.
    \end{equation*}
    then
    \begin{equation*}
        |\{u>t\} \cap B_1| \leq Ct^{-\varepsilon} \quad \text{for any } t>0.
    \end{equation*}
\end{theorem}

\begin{proof}
    We let $(\delta_0, \varepsilon_0, K, A_0)$ be the constants from Corollary~\ref{meas}. Then it suffices to show that there exists a constant $c=c(n) \in (0,1)$ such that
\begin{equation*}
    |\{ u > K^m\} \cap B_1| \leq (1-c\varepsilon_0 )^m |B_1| \quad \text{for any $m \in \mathbb{N}$}.
\end{equation*}
    In fact, in view of Proposition~\ref{twoprop} (with $r=1$ and $M=1$), $u$ satisfies either
    \begin{equation*}
        \mathcal{M}^-(D^2u)-\Lambda|Du| \leq f \ \text{ in } \{|Du| \leq A_0\} \cap B_{3}
    \end{equation*}
    or
    \begin{equation*}
        \mathcal{M}^-(D^2u)-\Lambda|Du| \leq f  \ \text{ in } \{|Du| \geq A_0+2\} \cap B_{3}.
    \end{equation*}
   Therefore, it follows from Corollary~\ref{meas} that 
   \begin{equation*}
       |\{ u > K\} \cap B_1| \leq (1-\varepsilon_0 ) |B_1|.
   \end{equation*}
    Setting $E_m =\{u>K^m\} \cap B_1$, it is immediate that $|E_{m+1}| \leq |E_1| \leq (1-\varepsilon_0)|B_1|$ for any $m \in \mathbb{N}$. We now suppose that $B=B_r(x_0) \subset B_1$ satisfies $|B \cap E_{m+1}| > (1-\varepsilon_0)|B|$. Then for the scaled function
    \begin{equation*}
        v(x)=\frac{u(rx+x_0)}{K^{m}},
    \end{equation*}
    we observe that
    \begin{equation*}
        |\{v>K\} \cap B_1|=\frac{|B \cap E_{m+1}|}{|B|}|B_1| >(1-\varepsilon_0)|B_1|.
    \end{equation*}
    In particular, $v$ satisfies all the assumptions of Corollary~\ref{meas} (with the aid of Proposition~\ref{twoprop} again), and so we obtain that
    \begin{equation*}
        v >1 \quad \text{in $B_1$, \quad or equivalently, \quad $B \subset E_m$}.
    \end{equation*}
    We now apply the growing ink-spots lemma (see \cite[Lemma~2.1]{IS16} for details) to conclude that
    \begin{equation*}
        |E_{m+1}| \leq (1-c\varepsilon_0)|E_m|,
    \end{equation*}
    which finishes the proof.
\end{proof}




Finally, we can prove the uniform H\"older regularity, Theorem~\ref{holder} (and so Theorem~\ref{lem-equicontinuity}), by repeating the standard arguments provided in \cite[Section~6]{IS16}, so we omit the proof.

\section{Improvement of flatness}\label{sec-improvement}
In this section, we prove the approximation lemma by using the compactness method from the uniform H\"older estimates (Theorem~\ref{lem-equicontinuity}). For simplicity, we write $\beta(x) = \beta(x, 0)$ when $x_0=0$.
\begin{lemma}[Approximation lemma]\label{lem-approximation}
    Suppose $f \in L^{n,1}(B_1) \cap C(B_1)$. 
    Let $q \in \mathbb{R}^n$ be an arbitrary vector and let $u \in C(B_1)$ be a viscosity solution of 
    \begin{equation*}
        |Du-q|^{\gamma}F(D^2u, x)=f(x)  \quad \text{in $B_1$},
    \end{equation*}
    satisfying $\|u\|_{L^{\infty}(B_1)} \leq 1$. Given any $\delta>0$, there exists $\eta>0$ depending only on $n$, $\lambda$, $\Lambda$, $\gamma$, $\delta$ such that if
    \begin{equation*}
        \|\beta\|_{L^{n,1}(B_1)} \leq \eta \quad \text{and} \quad \|f\|_{L^{n,1}(B_1)}  \leq \eta,
    \end{equation*}
    then there exists a function $h \in C(B_{1/2})$, solution to a constant coefficient, homogeneous, $(\lambda, \Lambda)$-uniformly elliptic equation
    \begin{equation}\label{eq-compactness}
        \mathfrak{F}(D^2h)=0 \quad \text{in $B_{1/2}$}
    \end{equation}
    such that
    \begin{equation*}
        \|u-h\|_{L^{\infty}(B_{1/2})} \leq \delta.
    \end{equation*}
\end{lemma}

\begin{proof}
    We proceed by contradiction. Suppose that there exists $\delta_0>0$ such that the statement fails.
    Then there exist a sequence of continuous functions $\{u_j\}$, a sequence of $(\lambda, \Lambda)$-uniformly elliptic operators $F_j : \mathcal{S}^n \times B_1 \to \mathbb{R}$, a sequence of vectors $\{q_j\} \subset \mathbb{R}^n$, and a sequence of functions $\{f_j\} \subset L^{n,1}(B_1)$ such that
    \begin{enumerate}[(i)]
        \item $\|u_j\|_{L^{\infty}(B_1)} \leq 1$;
        \item $|Du_j-q_j|^{\gamma}F_j(D^2u_j, x) =f_j$ in $B_1$;
        \item $\|\beta_j\|_{L^{n,1}(B_1)} +\|f_j\|_{L^{n,1}(B_1)}  \to 0 \quad \text{as $j \to \infty$}$;
    \item however,
    \begin{equation*}
        \sup_{B_{1/2}}|u_j-h| \geq \delta_0
    \end{equation*}
    for any $h$ satisfying \eqref{eq-compactness}.
    \end{enumerate}
    Here, $\beta_j$ denotes the modulus of oscillation of the coefficients of the operator $F_j$ as in \eqref{eq-def-beta}.
    Then by applying Theorem~\ref{lem-equicontinuity}, the sequence $\{u_j\}$ is pre-compact in $C(B_{2/3})$-topology.
    Thus, $u_j \to u_{\infty}$ locally uniformly (up to a subsequence) in $B_{2/3}$ for some $u_\infty \in C(B_{2/3})$. 

    It only remains to prove that the limit function $u_{\infty}$ is a solution of a constant coefficient, homogeneous, $(\lambda, \Lambda)$-uniformly elliptic equation \eqref{eq-compactness}. For this purpose, we divide our analysis into two cases.

    \textbf{Case 1:} If $|q_j|$ is bounded, then we can extract a subsequence of $\{q_j\}$ that converges to some $q_{\infty} \in \mathbb{R}^n$. Moreover, it follows from the uniform ellipticity that $F_j(\cdot, 0) \to \mathfrak{F}(\cdot)$ for some $(\lambda, \Lambda)$-uniformly elliptic operator $\mathfrak{F}$ up to a subsequence. We now claim that $u_{\infty}$ is a viscosity solution of    
    \begin{equation}\label{eq-claim}
        |Du_{\infty}-q_{\infty}|^{\gamma}\mathfrak{F}(D^2u_{\infty})=0.
    \end{equation}
    Once \eqref{eq-claim} is established, the cutting lemma (\cite[Lemma~6]{IS13}) yields that $u_{\infty}$ is a viscosity solution of $\mathfrak{F}(D^2u_{\infty})=0$, which contradicts (iv) for sufficiently large $j$.

    In order to prove \eqref{eq-claim}, we show that $u_{\infty}$ is a viscosity subsolution (the supersolution case is analogous).
    Let $P$ be a paraboloid that (strictly) touches $u_{\infty}$ by above in $B_r(x_0)$ at $x_0$, where $B_{2r}(x_0) \subset B_{1/2}$. Suppose, for the sake of contradiction, that
    \begin{equation*}
        |DP(x_0)-q_{\infty}|^{\gamma}\mathfrak{F}(D^2P)=-\eta <0.
    \end{equation*}
     By the continuity and the uniform convergence, and by choosing smaller $r$ if necessary, there exist constants $a>0$ and $A>0$ such that 
        \begin{equation*}
           a \leq  |DP(x)-q_{\infty}| \leq A \quad \text{for any $x \in B_r(x_0)$}
        \end{equation*}
        and
        \begin{equation*}
           a \leq  |DP(x)-q_{j}| \leq A \quad \text{for any $x \in B_r(x_0)$}
        \end{equation*}
    for all $j$ large enough.
    For $\varepsilon_j \coloneqq |\mathfrak{F}(D^2P)-F_j(D^2P, 0)|$, which vanishes as $j \to \infty$, we construct a corrector $\psi_j$ by solving the Dirichlet problem:
    \begin{equation*}
    \left\{
        \begin{aligned}
            \mathcal{M}^+_{\lambda^{\ast}, 1}(D^2\psi_j)&=-|f_j(x)|-|\beta_{j}(x)|-\varepsilon_j\eqqcolon g_j(x) &&\text{in $B_{2r}(x_0)$},\\
            \psi_j&=0 && \text{on $\partial B_{2r}(x_0)$}.
        \end{aligned}
        \right.
    \end{equation*}
    Since $g_j$ is continuous, the standard Perron's method guarantees the existence of a solution $\psi_j$. Since $\mathcal{M}^+$ is convex and $g_j$ is is assumed to be H\"older continuous \textit{a priori}, we have $\psi_j \in C^2(B_{2r}(x_0))$ and it satisfies
    \begin{equation*}
        \|(D^2\psi_j)^-\| \geq \frac{1}{n\lambda^{\ast}} \left(\|(D^2\psi_j)^+\|+|f_j|+|\beta_j|+\varepsilon_j \right).
    \end{equation*}
    Moreover, the borderline estimate (\cite[Theorem~1.2 and Theorem~2.1]{DKM14}) gives that $\psi_j$ satisfies the uniform estimate 
    \begin{equation*}
        \|\psi_j\|_{C^1(B_{r}(x_0))} \leq C(\|\psi_j\|_{L^{\infty}(B_{2r}(x_0))}+\|g_j\|_{L^{n,1}(B_{2r}(x_0))}) \leq C\|g_j\|_{L^{n,1}(B_{2r}(x_0))},
    \end{equation*}
    where the second inequality follows from the ABP estimate (see \cite[Theorem~3.2]{CC95} for example). Then it follows from the condition (iii) that
    \begin{equation*}
        \|\psi_j\|_{C^1(B_r(x_0))} \to 0 \quad \text{as $j \to \infty$}.
    \end{equation*}
    In particular, we have that for sufficiently large $j$, 
    \begin{equation*}
        \|D\psi_j\|_{L^{\infty}(B_{r}(x_0))} \leq a/2,
    \end{equation*}
    which implies
    \begin{equation*}
        |D(P+\psi_j)(x)-q_j| \geq |DP(x)-q_j|-|D\psi_j(x)| \geq a/2 \quad \text{in $B_{r}(x_0)$}.
    \end{equation*}
    Therefore, we arrive at
    \begin{equation*}
    \begin{aligned}
        F_j(D^2(P+\psi_j), x) &\leq F_j(D^2P, x)+\mathcal{M}^+_{\lambda, \Lambda}(D^2\psi_j) \\
        &\leq F_j(D^2P, 0)+\|D^2P\| \beta_j(x)+\Lambda\|(D^2\psi_j)^+\|-\lambda \|(D^2\psi_j)^-\|\\
        &\leq \mathfrak{F}(D^2P)+\varepsilon_j+\|D^2P\| \beta_j(x)+\Lambda\|(D^2\psi_j)^+\| \\
        &\qquad -\frac{\lambda}{n\lambda^{\ast}} \left(\|(D^2\psi_j)^+\|+|f_j(x)|+|\beta_j(x)|+\varepsilon_j \right)\\
        &\leq \mathfrak{F}(D^2P)-(2/a)^{\gamma}|f_j(x)| \, (<0) \quad \text{in $B_r(x_0)$}
    \end{aligned}
    \end{equation*}
    if we take $\lambda^{\ast}<1$ small enough. We next multiply $|D(P+\psi_j)(x)-q_j|^{\gamma} (\geq (a/2)^{\gamma})$ on both sides to obtain that
    \begin{equation}\label{eq-contradiction}
        \begin{aligned}
            |D(P+\psi_j)(x)-q_j|^{\gamma}F_j(D^2(P+\psi_j), x) &\leq (a/2)^{\gamma}\mathfrak{F}(D^2P)-|f_j(x)|\\
            &\leq (a/2A)^{\gamma} |DP(x_0)-q_{\infty}|^{\gamma}\mathfrak{F}(D^2P)-|f_j(x)|\\
            &=-(a/2A)^{\gamma}\eta-|f_j(x)|<f_j(x).
        \end{aligned}
    \end{equation}
    On the other hand, since $u_j \to u_{\infty}$ and $\psi_j \to 0$ uniformly in $B_r(x_0)$, and $P$ strictly touches $u_{\infty}$ by above at $x_0$ in $B_r(x_0)$, there exists a constant $C_j>0$ such that $P+\psi_j+C_j$ touches $u_j$ by above at some $x_j \in B_r(x_0)$ in $B_r(x_0)$. Thus, we have
    \begin{equation*}
        |D(P+\psi_j)(x_j)-q_j|^{\gamma}F_j(D^2(P+\psi_j), x_j) \geq f_j(x_j),
    \end{equation*}
    which directly contradicts \eqref{eq-contradiction} at the point  $x_j$.\newline

    \textbf{Case 2: }If $|q_j|$ is unbounded, then taking a subsequence, if necessary, we have $|q_j| \to \infty$. In this case, we define $e_j=q_j/|q_j|$ so that $u_j$ satisfies
    \begin{equation*}
        \left|\frac{Du_j}{|q_j|}-e_j\right|^{\gamma}F_j(D^2u_j, x)=\frac{f_j}{|q_j|^{\gamma}}.
    \end{equation*}
    By letting $j \to \infty$ and taking further subsequence, if necessary, we use the same
    argument with the previous case to find a limit function $u_{\infty}$ satisfying $\mathfrak{F}(D^2u_{\infty})=0$ for some $(\lambda, \Lambda)$-uniformly elliptic operator $\mathfrak{F}$. It again contradicts the condition (iv) for sufficiently large $j$, which finishes the proof.
\end{proof}

\begin{remark}
    For the proof of Lemma~\ref{lem-approximation}, the Lorentz-type conditions $\beta \in L^{n, 1}$ and $f \in L^{n, 1}$ were used for guaranteeing the $C^1$ estimate of the corrector $\psi_j$, so that $D\psi_j \to 0$ as $j\to \infty$. 
    
    On the other hand, there exists a function $f \in L^n(B_1)$ such that the Poisson equation $\Delta u=f$ admits a \emph{non-Lipschitz} solution $u$. In particular, one cannot expect a similar construction for the corrector $\psi_j$ to remain valid for the critical case $\beta \in L^n(B_1)$ and $f \in L^n(B_1)$. 
\end{remark}

\section{\texorpdfstring{$C^{1,\alpha}$}{C1a} regularity} \label{sec4}
In this section, we prove the optimal $C^{1, \alpha}$ regularity of viscosity solutions when $f \in L^p$ with $p>n$. We begin with the first step of the iteration argument.

\begin{lemma} \label{Capprolem}
    Suppose that $f \in L^p(B_1) \cap C(B_1)$ with $p>n$. 
    Let $q \in \mathbb{R}^n$ be an arbitrary vector and let $u \in C(B_1)$ be a viscosity solution of 
    \begin{equation*}
        |Du-q|^{\gamma}F(D^2u, x)=f(x)  \quad \text{in $B_1$},
    \end{equation*}
    satisfying $\|u\|_{L^{\infty}(B_1)} \leq 1$. Given 
    \begin{equation*}
        \alpha \in (0, \alpha_0) \cap \left(0, \frac{1-n/p}{1+\gamma}\right],
    \end{equation*}
    there exist constants $\rho \in (0, 1/4)$ and $\eta>0$ depending only on $n$, $p$, $\lambda$, $\Lambda$, $\gamma$ and $\alpha$ such that if
    \begin{equation*}
        \left(\int_{B_1}\beta^p(x) \, dx\right)^{1/p}\leq \eta \quad \text{and} \quad \left(\int_{B_1}|f(x)|^p \,dx\right)^{1/p} \leq \eta,
    \end{equation*}
    then there exists an affine function $l(x) \coloneqq a+{b} \cdot x$ satisfying
    \begin{equation*}
        \sup_{B_{\rho}}|u(x)-l(x)| \leq \rho^{1+\alpha}
    \end{equation*}
    and
    \begin{equation*}
        |a|+|{b}| \leq C(n, \lambda, \Lambda).
    \end{equation*}
\end{lemma}

\begin{proof}
    For $\delta>0$ to be determined later, let $h$ be a solution of a constant coefficient, homogeneous, $(\lambda, \Lambda)$-uniformly elliptic equation such that $\|u - h\|_{L^\infty(B_{1/2})} \leq \delta$. The existence of such a function is guaranteed by Lemma~\ref{lem-approximation} (since $\|f\|_{L^{n, 1}(B_1)} \lesssim \|f\|_{L^p(B_1)}$), provided that $\eta$ is chosen small enough depending on $\delta$ and universal parameters. Since our choice for $\delta$ will depend only on universal parameters, the choice of $\eta$ will be universal too.

    From the normalization of $u$, we have $\|h\|_{L^{\infty}(B_{1/2})} \leq 2$. Consequently, by the $C^{1, \alpha_0}$ regularity of $h$ (Theorem~\ref{thm-KS}), we obtain that
    \begin{equation*}
        \sup_{B_r}|h(x)-(h(0)+Dh(0)\cdot x)| \leq C(n, \lambda, \Lambda) \cdot r^{1+\alpha_0} \quad \text{for any $r \in (0, 1/4)$}
    \end{equation*}
    and
    \begin{equation*}
        |h(0)|+|Dh(0)| \leq C(n, \lambda, \Lambda).
    \end{equation*}
    Let us define
    \begin{equation*}
        l(x)=h(0)+Dh(0) \cdot x.
    \end{equation*}
    It immediately follows that
    \begin{equation*}
        \sup_{B_{\rho}} |u(x)-l(x)| \leq \delta+C(n, \lambda, \Lambda) \cdot \rho^{1+\alpha_0}.
    \end{equation*}
    Now, for a fixed exponent $\alpha<\alpha_0$, we choose $\rho$ and $\delta$ as
    \begin{equation*}
        \rho \coloneqq \left(\frac{1}{2C(n, \lambda, \Lambda)} \right)^{\frac{1}{\alpha_0-\alpha}} \quad \text{and} \quad \delta \coloneqq \frac{1}{2}\rho^{1+\alpha}.
    \end{equation*}
    Therefore, substituting these choices into the previous estimate yields
    \begin{equation*}
        \sup_{B_{\rho}} |u(x)-l(x)| \leq \rho^{1+\alpha}
    \end{equation*}
    as desired.
\end{proof}

\begin{proof}[Proof of Theorem~\ref{thm-main-p>n}]
By scaling, we may assume that $\|u\|_{L^\infty} \leq 1$ and $\|f\|_{L^p}\leq \eta$, where $\eta$ is the constant from Lemma~\ref{Capprolem}.  
We claim that there exists a sequence of affine functions $l_k(x)=a_k+b_k\cdot x$ satisfying
\begin{equation} \label{lkcond1}
    \sup_{B_{\rho^k}}|u(x)-l_k(x)| \leq \rho^{k(1+\alpha)},
\end{equation}
where $\rho\in(0, 1/4)$ is the radius from Lemma~\ref{Capprolem}.

We proceed by induction on $k$.
The base case $k=1$ is verified by Lemma \ref{Capprolem}.
Assuming that the claim holds for some $k$, we define
\begin{equation*}
    v(x)=\frac{(u-l_k)(\rho^kx)}{\rho^{k(1+\alpha)}}.
\end{equation*}
Then $|v| \leq 1$ in $B_1$ and $v$ satisfies 
\begin{equation*}
    |Dv-q_k|^\gamma F_k(D^2v,x) = f_k,
\end{equation*}
where
\begin{equation*}
    F_k(M,x) = \rho^{k(1-\alpha)}F({\rho^{-k(1-\alpha)}}M, \rho^kx), \quad q_k =-\rho^{-k\alpha} b_k \quad \text{and} \quad f_k(x) = \rho^{k(1-\alpha(1+\gamma))}f(\rho^kx).
\end{equation*}
Note that $F_k$ inherits the ellipticity of $F$ and $\beta_{F_k}(x) = \beta(\rho^kx)$.
Therefore, we obtain
\begin{equation*}
    \begin{aligned}
          \fint_{B_1}\beta_{F_k}^p(x) \, dx  &= \fint_{B_{\rho^k}}\beta^p(x) \, dx \, \leq \eta^p, \\
    \int_{B_1}|f_k(x)|^p \, dx  &=  \rho^{(p(1-\alpha(1+\gamma))-n)k} \int_{B_{\rho^k}}|f(x)|^p  \, dx \leq \eta^p,
    \end{aligned}
\end{equation*}
where we used the fact that $p(1-\alpha(1+\gamma))-n\geq0$.
Applying Lemma \ref{Capprolem} to $v$, there exists an affine function $\tilde{l}_k(x) =\tilde{a}_k+ \tilde{b}_k\cdot x$ such that 
\begin{equation*}
     \|v-\tilde{l}_k\|_{L^\infty(B_\rho)} \leq \rho^{1+\alpha} \quad  \text{ and} \quad |\tilde{a}_k|+|\tilde{b}_k| \leq C.
\end{equation*}
Setting $l_{k+1}(x) = l_k(x) + \rho^{k(1+\alpha)} \tilde{l}_k(\rho^{-k}x)$, we see that $l_{k+1}$ satisfies \eqref{lkcond1} for $k+1$, completing the induction.
Furthermore, the coefficients satisfy
\begin{equation*}
    |a_{k+1}-a_k| \leq C\rho^{k(1+\alpha)} \quad \text{and} \quad |b_{k+1}-b_k| \leq C\rho^{k\alpha},
\end{equation*}
which implies that $a_k$ and $b_k$ converge and $l_\infty =\lim_{k\rightarrow \infty} l_k$
becomes the affine approximation of $u$ at 0.
By the standard Schauder-type argument and covering argument, we conclude that $u$ is $C^{1,\alpha}$ in $B_{1/2}$ with the desired uniform estimate.
\end{proof}

\section{Log-Lipschitz regularity} \label{sec5}
In this section, we prove the log-Lipschitz regularity of viscosity solutions $u$, when the nonhomogeneous term $f$ belongs to the Lorentz space, i.e., $f \in L^{n, 1}(B_1)$.
\begin{lemma} \label{logapprolem}
    Let $u \in C(B_1)$ be a viscosity solution of 
    \begin{equation*}
            |Du-q|^\gamma F(D^2u,x) =f \quad \text{ in } B_1,
    \end{equation*}
    satisfying $\|u\|_{L^\infty(B_1)} \leq 1$.
    Then there exist $\rho \in (0, 1/4)$ and $\eta>0$ depending only on $n$, $\lambda$, $\Lambda$ and $\gamma$ such that if
    \begin{equation*}
        \|\beta\|_{L^{n, 1}(B_1)} \leq \eta \quad \text{and} \quad \|f\|_{L^{n, 1}(B_1)} \leq \eta,
    \end{equation*}
    then there exists an affine function $l(x) = a+b\cdot x$ such that 
    \begin{equation*}
        \|u-l\|_{L^\infty(B_{\rho})} \leq \rho
    \end{equation*}
    and
    \begin{equation*}
        |a|+|b| \leq C(n, \lambda, \Lambda).
    \end{equation*}
\end{lemma}

\begin{proof}
For given $\delta>0$, an application of Lemma~\ref{lem-approximation} guarantees the existence of a function $h \in C^{1,\alpha_0}_{\mathrm{loc}}(B_{1/2})$ which is a solution of the homogeneous equation and $\|u-h\|_{L^\infty(B_{1/2})} \leq \delta$.

Letting $l(x)= h(0) + Dh(0) \cdot x$, $C^{1,\alpha_0}$ estimates of $h$ (Theorem~\ref{thm-KS}) yield that
\begin{equation*}
    \sup_{B_\rho}|h(x)-l(x)| \leq C\rho^{1+\alpha_0} \quad \text{and} \quad |h(0)| + |Dh(0)| \le C.
\end{equation*}
Choosing $\rho \in (0, 1/4)$ small enough so that $C\rho^{1+\alpha_0} \leq \rho/2$ and $\delta = \rho/2$, we conclude the proof.
\end{proof}

\begin{proof}[Proof of Theorem~\ref{thm-main-p=n}]
By scaling, we may assume that $\|u\|_{L^\infty(B_1)} \leq 1$ and $\|f\|_{L^{n,1}(B_1)}\leq \eta$, where $\eta$ is the constant from Lemma~\ref{logapprolem}.  
Fix $y \in B_{1/2}$.
We claim that there exists a sequence of affine functions $l_k(x)=a_k+b_k\cdot (x-y)$ satisfying
\begin{equation} \label{lkcond2}
    \sup_{B_{\rho^k}(y)}|u(x)-l_k(x)| \leq \rho^k,
\end{equation}
where $\rho\in(0,1/4)$ is the radius from Lemma~\ref{logapprolem}.
We proceed by induction on $k$.
The base case $k=1$ follows directly from Lemma \ref{logapprolem}.
Assuming that the claim holds for some $k\geq1$, we consider
\begin{equation*}
    v(x)=\frac{(u-l_k)(y+\rho^kx)}{\rho^k}.
\end{equation*}
Then $|v| \leq 1$ in $B_1$ and $v$ satisfies 
\begin{equation*}
    |Dv-q_k|^\gamma F_k(D^2v,x) = f_k,
\end{equation*}
where
\begin{equation*}
    F_k(M,x) = \rho^kF(\rho^{-k}M, y+\rho^kx), \quad q_k =-b_k, \quad \text{and} \quad f_k(x) = \rho^kf(y+\rho^kx).
\end{equation*}
Note that $F_k$ inherits the ellipticity of $F$ and $\beta_{F_k}(x, 0) = \beta(y+\rho^kx,y)$.
Therefore, we have
\begin{equation*}
\begin{aligned}
    \|\beta_{F_k}\|_{L^{n,1}(B_1)} &= \rho^{-k}\|\beta(\cdot, y)\|_{L^{n,1}(B_{\rho^k}(y))} \leq \eta,\\
    \|f_k\|_{L^{n,1}(B_1)} &= \|f\|_{L^{n,1}(B_{\rho^k}(y))} \leq \eta.
\end{aligned}
\end{equation*}
Applying Lemma \ref{logapprolem} to $v$, there exists an affine function $\tilde{l}_k(x) =\tilde{a}_k+ \tilde{b}_k\cdot x$ such that 
\begin{equation*}
     \|v-\tilde{l}_k\|_{L^\infty(B_\rho)} \leq \rho \quad  \text{ and} \quad |\tilde{a}_k|+|\tilde{b}_k| \leq C.
\end{equation*}
We define $a_{k+1} = a_k + \rho^k\tilde{a}_k$ and $b_{k+1} = b_k + \tilde{b}_k$. Then $l_{k+1}(x) = a_{k+1} + b_{k+1} \cdot (x - y)$ satisfies \eqref{lkcond2} for $k+1$, completing the induction.
Moreover, we have
\begin{equation} \label{abcond}
    |a_{k+1}-a_k| \leq C\rho^k, \quad \text{and} \quad |b_{k+1}-b_k| \leq C.
\end{equation}
By \eqref{lkcond2} and \eqref{abcond}, $a_k$ converges to $u(y)$ with
\begin{equation*}
    |u(y)-a_k| \leq \frac{C}{1-\rho}\rho^k.
\end{equation*}
On the other hand, even though $b_k$ may not converge, we have
\begin{equation*}
    |b_k| \leq \sum_{j=0}^k|b_j-b_{j-1}| \leq Ck.
\end{equation*}
For $0<r<\rho$, let $k$ be the integer satisfying $\rho^{k+1} < r \leq \rho^k$.
Then
\begin{equation*}
    \begin{aligned}
         \sup_{x\in B_r(y)}|u(x)-u(y)| &\leq \sup_{x\in B_{\rho^k}(y)}|u(x)-l_k| + |u(y)-a_k| + |b_k|\rho^k \\
    &\leq C(\rho^k+k\rho^k) \\
    &\leq \frac{C}{\rho}\left(r+\frac{\log r}{\log \rho} r\right) \\
    &\leq -Cr\log r,
    \end{aligned}
\end{equation*}
where $C=C(n, \lambda, \Lambda, \gamma)$. This completes the proof. 
\end{proof}

\bibliographystyle{alpha}
\bibliography{literature}
\end{document}